\documentclass[preprint,number,12pt]{elsarticle}
\usepackage{graphicx}
\usepackage{subfigure}
\usepackage{tabu}
\usepackage{hyperref}
\usepackage{amssymb,color}
\hypersetup{
colorlinks=true,
linkcolor=blue, anchorcolor=green, citecolor=magenta, urlcolor=cyan, filecolor=magenta
}
\def\h{{\hspace{-0.2cm}}}

\usepackage{amssymb}
\usepackage{amsthm,amsmath}

\newtheorem{thm}{Theorem}
\newtheorem{cor}{Corollary}

\newtheorem{alg}{Algorithm}
\numberwithin{equation}{section}
\newcommand{\norm}[1]{ \left\| { #1 } \right\| }
\newcommand{\matt}[4]{{\left[ {\begin{array}{*{20}{c}}{#1} & {#2}\\{#3 }&{ #4}\end{array}} \right]} }
\newcommand{\vect}[2]{{\left[ {\begin{array}{*{20}{c}}{#1} \\{#2}\end{array}} \right]} }

\journal{Numerical Algorithms}

\begin{document}

\begin{frontmatter}



\title{A new relaxed HSS preconditioner for saddle point problems}


\author[Guilan]{Davod khojasteh Salkuyeh}\ead{khojasteh@guilan.ac.ir} \author[Guilan]{Mohsen Masoudi}\ead{masoudi\_mohsen@phd.guilan.ac.ir}

\address[Guilan]{Faculty of Mathematical Sciences, University of  Guilan,  Rasht, Iran}

\address{}

\begin{abstract}

We present a preconditioner for saddle point problems. The proposed preconditioner is extracted from a stationary iterative method which is convergent under a mild condition. Some properties of the preconditioner as well as the eigenvalues distribution of the preconditioned matrix are presented.
The preconditioned system is solved by a Krylov subspace method like restarted GMRES.
Finally, some numerical experiments on test problems arisen from finite element discretization of
the Stokes problem are given to show the effectiveness of the preconditioner.
\end{abstract}

\begin{keyword}
Saddle point problems\sep   HSS preconditioner \sep Preconditioning \sep Krylov subspace method \sep GMRES.

\MSC[2010]  65F08 \sep 65F10.
\end{keyword}

\end{frontmatter}

% \linenumbers

\section{Introduction}
\label{se1}
We study  the solution of the system of linear equations with the following block $2 \times 2$ structure
\begin{equation}
\label{saddle}
\mathcal{A}u= \matt{A}{B^T}{-B}{0}
\vect{x}{y}=\vect{f}{g}\equiv b,
\end{equation}
where $A \in \mathbb{R}^{n\times n}$ is a symmetric positive definite   matrix, $B \in  \mathbb{R}^{m \times n} $ with $rank(B)=m<n$. In addition, $x,f \in \mathbb {R}^n$, and $y,g \in \mathbb{R}^m$. We also assume that the matrices $A$ and $B$  are large and sparse. According to Lemma 1.1 in \cite{benzi 26 2004} the matrix $\mathcal A$ is nonsingular. Such systems
are called saddle point problems and appear in a variety of scientific and engineering problems; e.g., computational fluid dynamics, constrained optimization, etc. The readers are referred to \cite{bai 75 2006, benzi 14 2005} for more discussion on this subject.

 Several efficient iterative methods have been proposed during the recent decades to solve the saddle point problems \eqref{saddle}, such as SOR-like method \cite{golub 55 2001},  modified block SSOR iteration \cite{bai 10 1999,  bai 103 2001}, generalized SOR method \cite{bai 102 2005}, Uzawa method \cite{saad 2003}, parametrized inexact Uzawa methods \cite{bai 428 2008}, Hermitian and skew-Hermitian splitting (HSS) iteration methods \cite{bai 27 2007, bai 98 2004, bai 24 2003} and so on. However, in some situations, these iterative methods may be less efficient than the Krylov subspace methods \cite{saad 2003}. On the other hand, when  Krylov subspace methods are  applied to the saddle point problem \eqref{saddle}, tend to converge slowly.  But these methods can produce suitable preconditioners for accelerating the rate of convergence of the Krylov subspace methods.  In general, favourable rates of convergence of Krylov subspace methods are often associated with a clustering of most of the eigenvalues of preconditioned matrices around 1 and away from zero \cite{BenziSurvey}.
 In view of this, many preconditioners have been presented in literature, e.g.,  block diagonal preconditioners \cite{ murphy 21 2000, sturler 26 2005},  constraint preconditioners \cite{bai 31 2009, keller 21 2000},   block triangular preconditioners  \cite{bai 26 2004,cao 57 2007,simonchini 49 2004,wu 48 2008}, parametrized block triangular preconditioners  \cite{jiang 216 2010},  Hermitian and skew-Hermitian splitting (HSS) preconditioners  \cite{bai 98 2004,  benzi 26 2004,  pan 172 2006}.

In \cite{bai 24 2003}, Bai et al. proposed the HSS iteration method to solve  non-Hermitian positive definite  linear systems $Ax=b$ which  converges unconditionally to the unique solution of the system. For a given initial guess $x^{0}$, the HSS iteration can be written as
\begin{align}
\label{hssbasic}
{\left\{ {\begin{array}{*{20}{c}}
\left( \alpha  I + H \right) x^{k+\frac{1}{2}} &\h=\h&\left( \alpha  I -  S \right)  x^{k}+b,
\\
 \left( \alpha  I + S \right) x^{k+1} &\h=\h&\left( \alpha  I - H \right)  x^{k+\frac{1}{2}}+b,
\end{array}} \right.} \quad k=0,1,2,\ldots,
\end{align}
where $\alpha>0$ and $A=H+S$, in which $H=(A+A^*)/2$ and $S=(A-A^*)/2$, where $A^*$ denotes the conjugate transpose of $A$.
%The HSS iteration serves the preconditioner
%\[
%M=\frac{1}{2\alpha}(\alpha I+H)(\alpha I+S),
%\]
%for the system $Ax=b$.

Benzi and Golub in \cite{benzi 26 2004} have applied the HSS iteration method to the generalized saddle point problem (saddle point problems with nonzero $(2,2)$-block). As they  mentioned the convergence of the  method to solve the saddle point problem is typically too slow for the method to be competitive. For this reason they proposed using a nonsymmetric Krylov subspace method like the GMRES algorithm or its restarted version to accelerate the convergence of the iteration. Since the method has promising performance and elegant mathematical properties, it has attracted many researchers attention and many algorithmic variants and theoretical analysis of the HSS iteration for saddle point problems have been presented.  In \cite{bai 76 2007}, Bai et al. investigated the convergence properties of the  HSS iteration for the saddle point problem (\ref{saddle}) with $A$ being non-Hermitian and positive semidefinite. In \cite{benzi 61 2011}, Benzi and Guo proposed a dimensional split (DS) preconditioner for the Stokes and the linearized Navier-Stokes equations. The DS preconditioner is extracted from an HSS iteration method based on the dimensional splitting of ${\cal A}$.  A modification of the DS preconditioner has been presented by Cao et al. in \cite{YCaoDS2}. Benzi et al. have presented a relaxed version of DS in \cite{benzi 230 2011}. Some variants of the HSS preconditioner including its relaxed versions  have also been presented in the literature (see, e.g., \cite{Xie,FanZhengZhu,cao 31 2013}). In this paper, we present a new preconditioner which can be considered as a relaxed version of the HSS preconditioner for the saddle point problem.

Throughout the paper, for a matrix $X$, $\rho(X)$ and $X^*$ stand for the spectral radius and conjugate transpose of $X$, respectively. For a vector $x\in \Bbb{C}^n$, $\|x\|_2$ denotes the Euclidian norm of $x$. For a given matrix $A\in\Bbb{R}^{n\times n}$ and a vector $r\in\Bbb{R}^n$, the Krylov subspace $\mathcal{K}_m(A,r)$ is defined as $\mathcal{K}_m(A,r)=span\{r,Ar,\ldots,A^{m-1}r\}$.

This paper is organized as follows. In Section \ref{Sec2} we present our preconditioner. Some properties of the preconditioner are presented in Section \ref{Sec3}. Implementation of the proposed preconditioner is presented in Section \ref{Sec4}. Numerical experiments are given in Section \ref{Sec5}. The paper is ended by some concluding remarks in Section \ref{Sec6}.

\section{A review of the HSS preconditioner and its relaxed version}\label{Sec2}

In this section,  we first briefly review  the HSS iteration method and the induced HSS preconditioner for the saddle point problem. Then, a relaxed version of the HSS (RHSS) preconditioner, proposed by Cao et al.   in \cite{cao 31 2013}, is presented. Next, we give a new relaxed HSS (REHSS) preconditioner and investigate some of its properties.

\subsection{The HSS preconditioner for the saddle point problem}

According to the HSS iteration, the matrix $\cal A$ is split as
\[
\mathcal{A}=\mathcal{H}+\mathcal{S},\]
where
\[
\mathcal{H}=\frac{1}{2} \left( \mathcal{A}+\mathcal{A}^T\right)=\matt{A}{0}{0}{0}\quad \textrm{and} \quad
\mathcal{S}=\frac{1}{2} \left( \mathcal{A}-\mathcal{A}^T\right) = \matt{0}{B^T}{-B}{0}.
\]
Obviously, both of the matrices  $\alpha \mathcal I +\mathcal H$ and $\alpha \mathcal I + \mathcal S$ are nonsingular. In this case, the HSS iteration for the saddle point problem \eqref{saddle} is written as
\begin{align}
\label{hss}
{\left\{ {\begin{array}{*{20}{c}}
\left( \alpha  \mathcal I +\mathcal H \right) x^{k+\frac{1}{2}} &=&\left( \alpha \mathcal I - \mathcal S \right)  x^{k}+b,
\\
 \left( \alpha \mathcal I +\mathcal S \right) x^{k+1} &= &\left( \alpha \mathcal I -\mathcal H \right)  x^{k+\frac{1}{2}}+b.
\end{array}} \right.}
\end{align}
Computing $x^{k+\frac{1}{2}}$ from the first equation and substituting it in the second equation yields the iteration
\[
x^{k+1}={\Gamma} _{HSS} x^k+c,
\]
where
\[
 {\Gamma} _{HSS}=\left( \alpha \mathcal I + \mathcal S \right)^{-1}\left( \alpha \mathcal I - \mathcal H \right) \left( \alpha \mathcal I + \mathcal H \right)^{-1}\left( \alpha \mathcal I -\mathcal S \right),
\]
and
\[
c=2 \alpha\left( \alpha \mathcal I + \mathcal S \right)^{-1} \left( \alpha \mathcal I + \mathcal H \right)^{-1}b.
\]
It is known that  there is a unique splitting $\mathcal{A}={\mathcal{M}}_\alpha-{\mathcal{N}}_\alpha$, with ${\mathcal{M}}_\alpha$ being nonsingular, which induces the iteration matrix $ {\Gamma} _{HSS}$, i.e.,
\[
 {\Gamma} _{HSS}={\mathcal{M}}_\alpha^{-1} {\mathcal{N}}_\alpha=\mathcal I-{\mathcal{M}}_\alpha^{-1}\mathcal{A},
\]
where
\begin{align}
 { \mathcal{M}}_\alpha= \frac{1}{2\alpha}\left( \alpha \mathcal I + \mathcal H \right)\left( \alpha \mathcal I + \mathcal S \right), ~~
{\mathcal N}_\alpha=\frac{1}{2\alpha}\left( \alpha \mathcal I - \mathcal H \right)\left( \alpha \mathcal I -\mathcal S \right).
\end{align}

Benzi et al.  in \cite{benzi 26 2004} have shown  that for all $\alpha >0$, the HSS iteration is convergent unconditionally to the unique solution of the saddle point problem \eqref{saddle}.  As we know the HSS iteration serves the preconditioner ${\mathcal M}_\alpha$ for the system  \eqref{saddle} which is called the HSS preconditioner. Since the pre-factor $\frac{1}{2\alpha}$ in the HSS preconditioner ${\mathcal{M}}_\alpha$  has no effect on the preconditioned system, the HSS preconditioner can be written in the form
\begin{align}
\mathcal P_{HSS} & =\frac{1}{\alpha}\left( \alpha \mathcal I + \mathcal H \right)\left( \alpha \mathcal I + \mathcal S \right) =\frac{1}{\alpha} \matt{A+\alpha I}{0}{0}{\alpha I} \matt{\alpha I}{B^T}{-B}{\alpha I}\nonumber\\
&=\matt{A+\alpha I}{B^T+\frac{1}{\alpha}AB^T}{-B}{\alpha I}.
\end{align}
The difference between the HSS preconditioner $\mathcal P_{HSS}$  and the coefficient matrix $\mathcal A$ is
\begin{equation}\label{ErrHSS}
\mathcal{R}_{HSS}= \mathcal{P}_{HSS}-\mathcal{A}  =\matt{ \alpha I}{ \frac{1}{\alpha}AB^T}{0}{\alpha I}.
\end{equation}

\subsection{The RHSS preconditioner}

From Eq. \eqref{ErrHSS}, we see that as $\alpha$ tends to zero, the diagonal blocks tend to zero while the nonzero
off-diagonal block becomes unbounded. Hence, it is sought an appropriate  $\alpha$  to balance the weight of both parts. To do so,
Cao  et al. in \cite {cao 31 2013} consider the following relaxed HSS (RHSS) preconditioner  for the saddle point problem  \eqref{saddle}
\begin{equation}
\mathcal{P}_{RHSS}=\frac{1}{\alpha}\matt{A}{0}{0}{\alpha I}\matt{\alpha I}{B^T}{-B}{0}
=\matt{A}{\frac{1}{\alpha} AB^T}{-B}{0}.
\end{equation}
In this case, the difference between the RHSS preconditioner and the matrix $\mathcal A$ is given by
\begin{align}\label{ErrRHSS}
\mathcal{R}_{RHSS}=\mathcal{P}_{RHSS} -\mathcal{A}= \matt{0}{\left(\frac{1}{\alpha }A-I\right) B^T}{0}{0}.
\end{align}
Here, we see that as the parameter $\alpha$ tends to zero, the $(1,2)$-block of ${\mathcal R}_{RHSS}$ becomes unbounded.

From Eq. (\ref{ErrRHSS}), we have $\mathcal A=\mathcal{P}_{RHSS}-\mathcal{R}_{RHSS}$ which produces the RHSS iteration
\[
\mathcal{P}_{RHSS} x^{k+1}=\mathcal{R}_{RHSS} x^k+b, \quad k=0,1,\ldots,
\]
where $x^0$ is an initial guess.  Hence, the iteration matrix of the RHSS iteration is given by $\Gamma_{RHSS}=\mathcal{P}^{-1}_{RHSS}\mathcal{R}_{RHSS}$. In \cite {cao 31 2013}, it was shown that $\rho\left( \Gamma_{RHSS} \right) <1 $ for all $0<\alpha<\frac{2}{\mu_1}$ and the optimal value of $\alpha$ is
$\alpha _{opt}=2/(\mu_1+\mu_m)$, where $\mu_1$ and $\mu_m$ are, respectively,  the largest and smallest eigenvalues of the matrix $\left( BB^T\right)^{-1}\left( BA^{-1} B^T\right)$.

%%%%%%%%%%%%%%%%%%%%%%%%%%%%%%%%%%%%%%%%%%%%%%%%%%%%%%%%%%%%%%%%%%%%%%

\section{The REHSS preconditioner} \label{Sec3}

As we mentioned when  $\alpha $ tends to zero, the $(1,2)$-block in both of the matrices $\mathcal{R}_{HSS}$ and  $\mathcal{R}_{RHSS}$ become unbounded. To overcome this problem we consider the following  splitting for the matrix $\mathcal A$ as
\begin{equation}\label{RHSSSplit}
\mathcal{A}=\mathcal{P}_{REHSS}-\mathcal{R}_{REHSS}=
\matt{A}{AB^T}{-B}{\alpha I}-
\matt{0}{(A-I)B^T}{0}{\alpha I},
\end{equation}
where $\alpha >0$. As $\alpha$ tends to zero the $(2,2)$-block of $\mathcal{R}_{REHSS}$ tends to zero and in contrast with the HSS and the RHSS preconditioners the $(1,2)$-block remains bounded. This means that, for small values of $\alpha$ the REHSS preconditioner should be closer to the coefficient matrix $\mathcal{A}$ than the HSS and the RHSS preconditioners.

From the REHSS splitting \eqref{RHSSSplit} we state the REHSS iteration as
\[
\matt{A}{AB^T}{-B}{\alpha I} u^{k+1}= \matt{0}{(A-I)B^T}{0}{\alpha I}u^{k}+\vect{f}{g},
\]
to solve the saddle point problem \eqref{saddle}. In this case, the iteration matrix of the REHSS iteration is given by
\begin{equation}
\label{gama}
\Gamma_{REHSS}= \mathcal{P}_{REHSS}^{-1}\mathcal{R}_{REHSS}  =
\matt{A}{AB^T}{-B}{\alpha I}^{-1}
\matt{0}{(A-I)B^T}{0}{\alpha I}.
\end{equation}
The next theorem provides a sufficient condition for the convergence of the REHSS iteration.

%====================================Thm 1 =====================================

\begin{thm}
\label{th1}
Let $Q=B \left(  \frac{1}{2} A^{-1}- I \right)B^T$.  If $\delta=\lambda_{\max}(Q)$, then for every
$\alpha> \max \{\delta, 0\}$, it holds that  $\rho(\Gamma_{REHSS})<1$.
\end{thm}
\begin{proof}
We have
\[
\mathcal{P}_{REHSS}=\mathcal{M}_1 \mathcal{M}_2=
\matt{A}{0}{0}{I} \matt{I}{B^T}{-B}{\alpha I},
\]
where
\[
\mathcal{M}_2=\matt{I}{B^T}{-B}{\alpha I}\\
=\matt{I}{0}{-B}{I}\matt{I}{0}{0}{\alpha I+BB^T}\matt{I}{B^T}{0}{I}.
\]
Therefore
\begin{align}
	\label{inv}
	\mathcal{P}_{REHSS}^{-1}&=\mathcal{M}_2^{-1}\mathcal{M}_1^{-1} \nonumber \\
	&=\matt{I}{-B^T}{0}{I}\matt{I}{0}{0}{\left(\alpha I+BB^T\right)^{-1}}\matt{I}{0}{B}{I}\matt{A^{-1}}{0}{0}{I}\\
	&=\matt{A^{-1}-B^TS^{-1}BA^{-1}}{-B^TS^{-1}}{S^{-1}BA^{-1}}{S^{-1}}, \nonumber
\end{align}
where $S=\alpha I+BB^T$. Hence, we get
\begin{eqnarray}
	\label{preeig}
	\mathcal{P}_{REHSS} ^{-1}\mathcal{A}=\matt{I}{\tilde{A} }{0}{\hat{A}},
\end{eqnarray}
where $ \tilde{A}=A^{-1}B^T-B^TS^{-1}BA^{-1}B^T $ and $\hat{A}=S^{-1}BA^{-1}B^T$. As a result, we obtain
\begin{align*}
\Gamma_{REHSS}&=\mathcal P_{REHSS}^{-1} \mathcal R_{REHSS}= \mathcal P_{REHSS}^{-1}\left(\mathcal  P_{REHSS}- \mathcal A\right)\\
 &=I-\mathcal P_{REHSS}^{-1}\mathcal A=\matt{0}{-\tilde{A}}{0}{I-\hat{A}}.
\end{align*}
Hence, if $\lambda $ is an eigenvalue of the matrix $\Gamma_{REHSS}$, then $ \lambda=0 $ or $ \lambda=1-\mu$, where $\mu$ is an eigenvalue of the matrix $\hat{A}$. Therefore, there exists a vector $x \neq 0$ such that
\[
\hat{A}x=\left(\alpha I+BB^T\right)^{-1}BA^{-1}B^T x  = \mu x.
\]
Without loss of generality, we assume that $\norm{x}_2=1$. Since $B^Tx \neq 0$, we have
\[
\mu=\frac{x^*BA^{-1}B^T x}{\alpha +x^*BB^T x }>0.
\]
Hence, $| \lambda |<1$ if and only if
\[
 \frac{ x^*BA^{-1}B^T x}{\alpha +x^*BB^T x }<2.
\]
which  is equivalent to
\begin{equation}\label{ALR}
\alpha> x^*B\left( \frac{1}{2} A^{-1}- I \right)B^Tx=x^*Qx.
\end{equation}
Therefore, a sufficient condition to have $|\lambda|<1$ is
\[
\alpha > \max _{\norm{x}_2=1}   x^*Qx =\lambda_{\max}(Q)=\delta.
\]
It is necessary to mention that the matrix $Q$ is symmetric and hence all of its eigenvalues are real.
\end{proof}

%============================Corollary ===================================================================
\begin{cor}
Assume that
\begin{equation}\label{Conconv}
\lambda_{\min}(A)> \frac{1}{2}\kappa(B)^2,
\end{equation}
where $\kappa(B)$ and  $\lambda_{\min}(A)$ stand for the spectral condition number and smallest eigenvalue of $A$. Then, for every $\alpha>0$, it holds that $\rho (\Gamma_{REHSS})<1$.
\begin{proof}
From  \cite[Theorem 1.22]{saad 2003} we have
\begin{eqnarray*}
x^*B A^{-1}B^T x &\h \leq \h& \lambda_{\max}(A^{-1}) x^*BB^Tx \leq \frac{1}{\lambda_{\min}(A)} \lambda_{\max}(BB^T) x^*x= \frac{\sigma_{\max}(B)^2}{\lambda_{\min}(A)},\\
x^*BB^Tx &\h \geq \h& \lambda_{\min}(BB^T)x^*x=\sigma_{\min}(B)^2,
\end{eqnarray*}
where $\sigma_{\min}(B)$ and $\sigma_{\max}(B)$ stand for the smallest and largest singular values of $B$. From these inequalities and  Eq. \eqref{ALR} we deduce
\[
x^*Qx=\frac{1}{2} x^*B  A^{-1}B^Tx- x^*BB^Tx \leq \frac{1}{2} \frac{\sigma_{\max}(B)^2}{\lambda_{\min}(A)}-\sigma_{\min}(B)^2=\theta.
\]
It follows from this equation that  $\delta\leq \theta$, where $\delta$ was defined in Theorem  \ref{th1}.  Hence, if $\alpha > \max\{0,\theta\}$, then the convergence of the REHSS iteration is achieved. Now, if $\theta<0$ then for every $\alpha>0$ we get $\rho (\Gamma_{REHSS})<1$. Obviously, $\theta<0$ is equivalent to the condition \eqref{Conconv}.
\end{proof}
\end{cor}

%==============================================Theorem 2==========================================================================
The next theorem analyses the behavior of $\mathcal{P}_{REHSS}^{-1} \mathcal{A}$.

\begin{thm}
(a) For $\alpha >0$, the preconditioned matrix $\mathcal P_{REHSS}^{-1}\mathcal{A} $ has eigenvalue 1 of algebraic multiplicity at least $n$. The remaining eigenvalues are $\mu_i $, where  $\mu_i $ are the eigenvalues of the $m \times m$  matrix $\hat{A}=\left( \alpha I+BB^T\right)^{-1}BA^{-1}B^T $.\\
(b) Let $(\mu,[x;y])$ be an eigenpair of $\mathcal{P}_{REHSS}^{-1} \mathcal{A}$. Then, $x\neq 0$ and $\mu$ is either equal to $1$ or can be written
as  $\mu=(\alpha \hat b+\hat c)/\hat a$, where
\[
\hat a= x^*\left(\alpha I + B^T B \right) A  \left(\alpha I+ B^TB\right)x,\quad
\hat b=x^*\left(  B^T B \right)x,\quad \hat c=x^*\left(  B^T B \right)^2x.
\]
Moreover, when $\alpha\rightarrow 0$, then $\mu$ is either equal to $1$ or
\[
\frac{1}{\mu_{\max}(A)} \leqslant \mu \leqslant \frac{1}{\mu _{\min}(A)},
\]
where $\mu_{\min}(A)$ and $\mu_{\min}(A)$ are the smallest and largest eigenvalues of $A$, respectively.
\end{thm}
\begin{proof}
Part (a) follows immediately from Eq. \eqref{preeig}. To prove (b), let $(\mu,[x;y])$ be an eigenpair of $\mathcal{P}_{REHSS}^{-1} \mathcal{A}$. Therefore,
\[
\mathcal{A}\vect{x}{y} =\mu \mathcal{P}_{REHSS}\vect{x}{y},
\]
which is equivalent to
\begin{align}
%{\left\{ {\begin{array}{*{20}{c}}
Ax+B^Ty&=\mu Ax+\mu AB^T y, \nonumber \\
-Bx&=-\mu B x+\mu\alpha y.\label{eq2}
%\end{array}} \right.}
\end{align}
Hence
\begin{align}
\label{eq3}
\left( \mu-1 \right)Ax+\left( \mu A-I \right) B^Ty=0,\\
\label{eq4}
\mu \alpha y=\left( \mu-1 \right) Bx.
\end{align}
Premultiplying both sides of \eqref{eq4} by $B^T$ yields
\begin{align}
\label{eqq5}
\mu \alpha B^Ty=\left( \mu-1 \right)B^TBx.
\end{align}
Multiplying both sides of \eqref{eq3}  by $\mu \alpha$ and substituting   \eqref{eqq5} in it, gives
\[\mu \alpha \left( \mu-1 \right) Ax+\left( \mu A-I \right) \left( \mu-1 \right)B^TBx=0.
\]
We show that $x\neq 0$. Otherwise,  from \eqref{eq2} we get $\mu=0$ or $y=0$. In fact, neither of them can be zero. So $x\neq0$. Without loss of generality, let $\norm{x}_2 =1$. Hence
\begin{align}
\label{eq6}
\mu^2 A  \left(\alpha I + B^TB\right)x -\mu  \left(A \left(\alpha I + B^T B\right)    + B^TB \right ) x+B^T B x=0.
\end{align}
Multiplying $x^*\left(\alpha I+ B^T B \right)$  to both sides of  \eqref{eq6}, yields
  \[
   \hat a \mu^2- \left(\hat a+\alpha \hat b + \hat c \right)\mu+ \left( \alpha\hat b+ \hat c\right)=0.
  \]
The roots of this quadratic equation are $\mu=1$ and
\[
\mu=\frac{\alpha \hat b+\hat c}{\hat a}= \frac{\alpha \hat b+\hat c}{\alpha^2 x^*Ax+ \alpha \left( x^* B^ T B Ax +x^* A   B^T Bx\right)  + x^* B^T B A B^T B x}.
\]

To prove  the last part of theorem, we show that if $Bx=0$, then $\mu=1$. If $Bx=0$, then it follows from Eq. \eqref{eq4} that $y=0$. Substituting this in Eq. \eqref{eq3}, yields $(\mu-1)Ax=0$. Now, since $Ax\neq 0$, we conclude that $\mu=1$.
Therefore, if $\alpha \rightarrow 0$, then $\mu =1$ or
\[
\mu =\frac{x^*\left( B^T B\right)^2 x}{x^* B^T B A B^T B x} = \frac{ (B^T Bx)^*(B^T Bx)}{  (B^TBx)^*  A   (B^T B x)}=\frac{z^*z}{z^*Az},
\]
where $z=B^TBx$. Since, $A$ is symmetric positive definite we have
\[
\frac{1}{\mu_{\max}(A)} \leqslant \frac{z^*z}{z^*Az} \leqslant \frac{1}{\mu _{\min}(A)},
\]
which completes the proof.
\end{proof}

%======================================================================================================

\begin{thm}
The degree of the minimal polynomial of the preconditioned matrix $\mathcal{P}_{REHSS}^{-1} \mathcal{A}$
 is at most $m + 1$. Thus, the dimension of the
Krylov subspace $\mathcal{K}_n(\mathcal{P}_{REHSS}^{-1} \mathcal{A}, b)$ is at most $m + 1$.
\end{thm}
\begin{proof}
Let $\chi $ be the characteristic
polynomial of the preconditioned matrix $\mathcal{P}_{REHSS}^{-1} \mathcal{A}$. By using \eqref{preeig}, we have
\begin{align*}
\chi(x)=\left(x-1 \right)^{n}\prod_{i=1} ^m \left( x-\mu_i \right),
\end{align*}
where $\mu_i$,  for $i=1,\ldots,m$, are the eigenvalues of the matrix $\hat{A}$. Let
\[p(x)=(x-1) \prod_{i=1} ^m \left( x-\mu_i \right).\]
Therefore
\begin{align*}
p(\mathcal{P}_{REHSS}^{-1}\mathcal{ A})&=\left(\mathcal P_{REHSS}^{-1} \mathcal{A}-\mathcal I\right) \prod_{i=1} ^m \left( \mathcal {P}_{REHSS}^{-1} \mathcal{A}-\mu_i \mathcal I \right)\\
&=\matt{0}{\tilde{A}}{0}{\hat{A}-I_{m}}  \prod_{i=1} ^m\matt{(1-\mu_i)I_n}{\tilde{A}}{0}{\hat{A}-\mu_i I_m} \\
&=\matt{0}{\tilde{A} ~\displaystyle\prod_{i=1} ^m (\hat{A}-\mu_i I_m)}{0}{( \hat{A}-I_{m} ) \displaystyle\prod_{i=1}^m (\hat{A}-\mu_i I_m ) }.
\end{align*}
Since for $i=1,\ldots, m$, $\mu_i$ is an eigenvalue of the matrix $\hat{A}$, we have $$\displaystyle\prod_{i=1} ^m (\hat{A}-\mu_i I_m)=0$$ and so $p(\mathcal {P}_{REHSS}^{-1}\mathcal{ A})=0$. Therefore the degree of the minimal polynomial of the preconditioned matrix $\mathcal{P}_{REHSS}^{-1} \mathcal{A}$  is at most $m + 1$. Hence, by \cite[Proposition 6.1] {saad 2003}, the dimension of the Krylov subspace $\mathcal{K}_{n}(\mathcal {P}_{REHSS}^{-1} \mathcal{A}, b)$ is also at most $m + 1$.
\end{proof}

\section{Implementation of $P_{REHSS}$}\label{Sec4}

We use the restarted version of the GMRES (denoted  by GMRES($m$)) in conjunction with the preconditioner  $P_{REHSS}$ to solve the saddle point problem \eqref{saddle}. At each step of applying the preconditioner  $P_{REHSS}$ within the GMRES($m$) algorithm
we need to compute a vector of the form $z=P_{REHSS}^{-1}r$  for a given vector  $r=[r_1;r_2]$ where $r_1\in \Bbb{R}^n$ and $r_2\in\Bbb{R}^m$. Let $z=[z_1;z_2]$, where $z_1\in \Bbb{R}^n$ and $z_2\in\Bbb{R}^m$. Now, form Eq. \eqref{inv} we can compute
the vector $z$ via
\[
\vect{z_1}{z_2}=\matt{I}{-B^T}{0}{I}\matt{I}{0}{0}{\left(\alpha I+BB^T\right)^{-1}}\matt{I}{0}{B}{I}\matt{A^{-1}}{0}{0}{I}\vect{r_1}{r_2}.
\]
We can use Algorithm 1 to compute the vector $z$.
\begin{alg}
\label{algpre} Computation of $z=P_{REHSS}^{-1}r$.
\begin{enumerate}
\item Solve $Aw_1=r_1$ for $w_1$.
\item
Solve $( \alpha I+BB^T) w_2=Bw_1+r_2$ for $w_2$.
\item
$z_2:=w_2$.
\item
$z_1:=w_1-B^Tw_2$.
\end{enumerate}
\end{alg}

Both of the matrices $A$ and $\alpha I+BB^T$ are symmetric positive definite. Hence, we can solve the systems appeared in steps 1 and 2 of Algorithm \ref{algpre} exactly by the Cholesky factorization or approximately by the conjugate gadient (CG) or the  preconditioned conjugate gradient (PCG) iterative method.

\section{Numerical experiments}\label{Sec5}

In this section, we present some numerical experiments to illustrate the effectiveness of the
preconditioner $\mathcal P_{REHSS}$  for the saddle point problem \eqref{saddle}.
The restarted GMRES method \cite{saad 2003} with restarting frequency 30, i.e., GMRES (30), is appled to the left preconditioned saddle point problem \eqref{saddle} in conjunction with the preconditioner  $\mathcal P_{REHSS}$ and the corresponding numerical results are compared with those of the preconditioners  $\mathcal P_{HSS}$ and $\mathcal P_{RHSS}$ in terms of  iteration counts and CPU timings. All runs are performed in \textsc{Matlab} 2014 on an Intel core i7 (12G RAM) Windows 8 system.

In all the tests, the initial vector is set to be a zero vector and the right-hand side vector $b=[f; g] \in \mathbb{R}^{n+m}$ is chosen  such that the exact solution of the saddle point problem \eqref{saddle} is a vector of all ones. We use the stopping criterion
\[
 \norm{\mathcal P r_k}_2 \leqslant 10^{-12} \norm{ \mathcal P  b}_2,
\]
where $r_k = b-\mathcal Au_k$ is the residual at the $k$th iteration and $\mathcal P$ is one of the preconditioners $\mathcal P_{HSS}$, $\mathcal P_{RHSS}$ or $\mathcal P_{REHSS}$. The maximum number of the iterations and the maximum elapsed  CPU time are set to be $k_{\max} = 500$ and  $t_{\max}=3600s$, respectively.  Throughout this section,  ``IT" and ``CPU" stand for the numbers of the restarts in GMRES($m$) and the CPU time, respectively. In all the tables, a dagger ($\dag$) shows that the method has not converged in at most $k_{\max}$ iterations. Similarly, a $``\ddag"$ shows that the method has not converged after elapsing $t_{\max}$ seconds.  At each step of applying  the preconditioners $\mathcal P_{HSS}$, $\mathcal P_{RHSS}$, and $\mathcal P_{REHSS}$, we need to solve two sub-systems with symmetric positive definite   coefficients matrix  (see Algorithm \ref{algpre} and  \cite[Algorithm 3.3 and Algorithm 3.4]{cao 31 2013}) and  all of these systems are solved  by the Cholesky  factorization.

Consider the Stokes problem
(see  \cite{elman 33 2007} or \cite[page 221]{Elman})
\begin{equation}\label{stokes}
\left\{\begin{array}{ll}
-\triangle \textbf{u}+\nabla p=\textbf{f}, \\
\hspace{1.1cm}\nabla . \textbf{u}=0, \\
\end{array}\right.
\end{equation}
in $\Omega=[-1,1] \times [-1,1]$, where $\Delta$, $\nabla$,  $\mathbf {u}$, and  $\mathbf p$ stand for the Laplace operator, the gradient operator, velocity and pressure of the fluid, respectively, with suitable boundary condition on $\partial \Omega$. It is known that many discretization schemes for \eqref{stokes} will lead to
saddle point problems of the form \eqref{saddle}.
 We consider Q2-P1 finite element discretizations on uniform grids on the unit square of the  tree  standard model problems (see \cite{elman 33 2007,Elman})
 \begin{enumerate}
  \item  The leaky lid-driven cavity problem;
  \item  The channel domain problem;
  \item  The  colliding flow problem.
 \end{enumerate}

We use the IFISS software package developed by Elman et al. \cite{elman 33 2007} to generate the linear systems  corresponding to $16 \times 16$, $32 \times 32$, $64 \times  64$, $128 \times 128$, and $256 \times 256$ meshes.  The IFISS software provides the matrices \verb"Ast" and \verb"Bst" for the matrices $A$ and $B$, respectively. For the channel domain problem the matrix $\verb"Bst"$ is of full rank, but for the colliding flow  and lid driven cavity problem is rank deficient. Therefore, in these cases  we drop two first rows of \verb"Bst" to get a full rank matrix.
Generic information of the test problems, including $n$, $m$, $nnz(A)$ and $nnz(B)$, are given in Table \ref{Tablesize} where $nnz$ stand for the number of the nonzero entries of a matrix. We present the numerical results for different values of $\alpha$ ($\alpha=10^{-6},10^{-4},10^{-2},1,10^{2}$).

\begin{table}[!t]
\begin{center}
 \caption{The size of the matrices $A$ and $B$ for different of   grids. }
 \label{Tablesize}
\scriptsize
\begin{tabular}
{cccccccccccccccccccccccccccccc}\\ \hline
&    \multicolumn{4}{c}{channel domain problem} &&  \multicolumn{4}{c}{lid driven cavity and cooliding flow problem}  \\
\cline{2-5} \cline{7-10} \\
 {Grid}  & $n  $&   $ m$   &$nnz(A)$&$ nnz(B)$&& $n$  &   $ m$   &$nnz(A)$& $nnz(B)$\\ \hline
%%%%%%%%%%%%%%%%%%%%%%%%%%%%%%%%%%%%%
{$16 \times 16 $}& 578 &192 & 6698  &  2084&  & 578& 190 & 6178 & 1967\\
{$32 \times 32$}&   2178   &768  & 29546 & 10142 & & 2178& 766  &   28418 & 9868\\
{$64 \times 64 $}&  8450 & 3072 &  124550 & 45062 & & 8450 & 3070  &  122206 & 44516 \\
{$128 \times 128 $}& 33282   &  12288  &    511152    &  192174&& 33282  &  12286     & 506376    &  191084 \\
{$256 \times 256 $}& 132098   &    49152 &    2070764  &    791738&& 132098  &  49150&     2061140   &   789560\\
         \hline
\end{tabular}\end{center}
\end{table}

Numerical results for the leaky lid-driven cavity,  the channel domain , and the  colliding flow problems are, respectively, presented in Tables \ref{tabliddriven}, \ref{tabchannel} and \ref{tabcolliding}. In all the tables ``IT" stands for the number of restarts in the GMRES(30) algorithm and ``CPU" denotes the elapsed CPU time for the convergence. As the numerical results show almost for all the three test problems the preconditioner  $\mathcal{P}_{REHSS}$  is more effective than the preconditioners $\mathcal {P}_{HSS}$ and $\mathcal {P}_{RHSS}$ in terms of the iteration counts and CPU time. The exceptions are the test problems with $\alpha=10^{-4}$ and $r=5,6$ (see Tables \ref{tabliddriven}, \ref{tabchannel} and \ref{tabcolliding}) where the results of the $P_{RHSS}$ preconditioner are slightly better than those of the $P_{REHSS}$ preconditioner. As we see, the GMRES(30) method for the preconditioned system with preconditioner $\mathcal{P}_{REHSS}$ always converges, whereas it does not converge for other two preconditioners. Another observation which can be posed here is that, despite preconditioners $\mathcal {P}_{HSS}$ and $\mathcal {P}_{RHSS}$, the behavior of the preconditioned iteration corresponding to the preconditioner $\mathcal{P}_{REHSS}$ is not very sensitive to the choice of $\alpha$.

\begin{table}[!h]
\begin{center}
 \caption{Numerical results lid driven cavity problem  on   $2^r \times 2^r$ grid. }\label{tabliddriven}
\scriptsize
\begin{tabular}
{cccccccccccccccccccc}\\ \hline
{}  & {} &    \multicolumn{2}{c}{$\alpha=10^{-4}$} & &  \multicolumn{2}{c}{$\alpha=10^{-2}$}  & & \multicolumn{2}{c}{$\alpha=1 $}  &  &\multicolumn{2}{c}{$\alpha=10^2$} \\
\cline{3-4} \cline{6-7}  \cline{9-10}  \cline{12-13} \\
{$r$} & {Preconditioner}   & {IT}            & {CPU}&& {IT}    & {CPU}&& {IT}    & {CPU}&& {IT}    & {CPU} \\ \hline
%%%%%%%%%%%%%%%%%%%%%%%%%%%%%%%%%%%%%%%
    &$\mathcal P_{HSS}$  &   4&    0.04&&  5&  0.06&& 13& 0.16&&106&    1.46\\
$4 $  &$\mathcal P_{RHSS}$  &  3&    0.02&&  3& 0.02&&  3& 0.02&&  4&    0.04\\
     &$\mathcal P_{REHSS}$  &  3&    0.02&&  3& 0.02&&  3& 0.02&&  3&    0.02  \\
        \hline
%%%%%%%%%%%%%%%%%%%%%%%%%%%%%%
   &$\mathcal P_{HSS}$  &  8&    0.37&&  9& 0.45&&144& 8.45&&$\dag$&   -\\
$5 $  &$\mathcal P_{RHSS}$  &  5&    0.19&&  5& 0.20&&  5& 0.21&&  9&    0.43\\
   &$\mathcal P_{REHSS}$  &5&    0.21&&  4& 0.12&&  3& 0.07&&  3&    0.07\\
        \hline
%%%%%%%%%%%%%%%%%%%%%%%%%%%%%%%%%%%%%%
   &$\mathcal P_{HSS}$  & 14&    3.87&& 47& 14.21&&$\dag$& -&&$\dag$&   -\\
$6 $  &$\mathcal P_{RHSS}$  &8&    1.97&&  8&  2.03&&  9&  2.29&& 27&     8.10\\
   &$ P_{REHSS}$  & 11&    3.08&&  3& 0.47&&  3& 0.42&&  3&    0.39\\
        \hline
%%%%%%%%%%%%%%%%%%%%%%%
     &$\mathcal P_{HSS}$  & 38&   76.32&&$\dag$&-&&$\dag$& -&&$\dag$& -\\
$7 $  &$\mathcal P_{RHSS}$  & 15&   30.85&& 14&  25.95&& 17&  33.46&& 79&  160.91\\
    &$\mathcal P_{REHSS}$  & 9&    17.30&&  3&  3.86&&  3&  3.69&&  3&    3.86\\
        \hline
%%%%%%%%%%%%%%%%%%%%
 &$\mathcal P_{HSS}$  &115& 2064.44&&- &  $\ddag$&& -&   $\ddag$&& -&  $\ddag$\\
$8 $  &$\mathcal P_{RHSS}$  &37&   641.50&& 28&  471.97&& 38& 648.74&&-&  $\ddag$\\
  &$\mathcal P_{REHSS}$  &5&   67.63&&  3&  33.11&&  3&  31.91&&  3&   27.43\\
        \hline \\[1mm]
\end{tabular}
\end{center}
\end{table}

\begin{table}[!h]
\begin{center}
 \caption{Numerical results for channel domain problem on $2^r \times 2^r$ grid.
 }\label{tabchannel}
\scriptsize
\begin{tabular}
{ccccccccccccccccccccc}\\ \hline
{}  & {}  &   \multicolumn{2}{c}{$\alpha=10^{-4}$} & &  \multicolumn{2}{c}{$\alpha=10^{-2}$}  & & \multicolumn{2}{c}{$\alpha=1 $}  &  &\multicolumn{2}{c}{$\alpha=10^2$} \\
\cline{3-4} \cline{6-7}  \cline{9-10} \cline{12-13} \\
{$r$} & {Preconditioner}   & {IT} & {CPU}&& {IT}    & {CPU}&& {IT}    & {CPU}&& {IT}    & {CPU} \\ \hline
%%%%%%%%%%%%%%%%%%%%%%%%%%%%%%%%%%%%%%%
       &$\mathcal P_{HSS}$   &  5& 0.05&&  6&   0.06&&  7&  0.07&& 17& 0.21\\
$4 $  &$\mathcal P_{RHSS}$ &3& 0.02&& 3&   0.03&&  3&  0.02&&   4& 0.03\\
        &$\mathcal P_{REHSS} $ &3&0.02&&  3&   0.02&&  3&  0.02&&  3& 0.02\\
        \hline
%%%%%%%%%%%%%%%%%%%%%%%%%%%%%%%%%
       &$\mathcal P_{HSS}$ & 9& 0.48  &&10  & 0.50&&13&  0.72&&47& 2.84 \\
$5$  &$\mathcal P_{RHSS}$ & 5&  0.22&&5&  0.22&& 5&   0.20&& 9&0.47\\
       &$\mathcal P_{REHSS} $ &  5& 0.21&&3&  0.09&& 3&  0.07&& 3& 0.06\\
        \hline
%%%%%%%%%%%%%%%%%%%%%%%%%%%%%%%%%
       &$\mathcal P_{HSS}$ &21&5.94&& 13&  3.46&&28&  8.39&& 498 & 154.59\\
$6$  &$\mathcal P_{RHSS}$ & 8& 2.08&& 8&  2.06&& 8&  1.96&&21&6.15\\
       &$\mathcal P_{REHSS}$ & 6&1.45&&3&   0.40&& 3&   0.40&& 3&0.38\\
        \hline
%%%%%%%%%%%%%%%%%%%%%%%%%%%%%%%%%%%%%%%
      &$\mathcal P_{HSS}$ &48&   96.57&&15&  28.06&& 83& 169.45 &&320&  1089.81\\
$7$  &$\mathcal P_{RHSS}$  &  17& 32.79 && 16&  30.53&& 15&  28.83 && 86 &178.10\\
      &$\mathcal P_{REHSS}$ & 5& 7.45&& 3&  2.84&& 3&  2.65&& 3 & 2.52\\
       \hline
%%%%%%%%%%%%%%%%%%%%%%%%%%%%%%%%%%%%%%%%
      &$\mathcal P_{HSS}$  &169&3093.67&& 20&336.76&&-&$\ddag$&& - &   $\ddag$\\
$8$  &$\mathcal P_{RHSS}$  &  42&  758.74&& 37 & 654.38 && 34& 603.85&&
 -&  $\ddag$\\
   &$\mathcal P_{REHSS}$  &  4&40.52&& 3& 24.29&& 3& 22.11&& 2& 16.03 \\         \hline
\end{tabular}
\end{center}
\end{table}

\begin{table}[!h]
\begin{center}
 \caption{Numerical results for the colliding flow problem on  $2^r \times 2^r$ grid.
 }\label{tabcolliding}
\scriptsize
\begin{tabular}
{cccccccccccccccccccccc}\\ \hline \\
{}  & {}  &   \multicolumn{2}{c}{$\alpha=10^{-4}$} & &  \multicolumn{2}{c}{$\alpha=10^{-2}$}  & & \multicolumn{2}{c}{$\alpha=1 $}  &  &\multicolumn{2}{c}{$\alpha=10^2$} \\
\cline{3-4} \cline{6-7}  \cline{9-10} \cline{12-13}  \\
{$r$} & {Preconditioner} &{IT}& {CPU}& & {IT}    & {CPU}&& {IT}    & {CPU}&& {IT}    & {CPU} \\ \hline
%%%%%%%%%%%%%%%%%%%%%%%%%%%%%%%%%%%%%
      &$\mathcal P_{HSS}$ & 4& 0.04&&   5& 0.05&&    13&  0.16&&   90&   1.17\\
$4 $  &$\mathcal P_{RHSS}$ &3& 0.02&&    3& 0.02&&    3&  0.02&&  4&   0.03\\
      &$\mathcal P_{REHSS}$ &3& 0.02&&     3& 0.02&&     3&  0.02&&  3&   0.02\\ \hline
%%%%%%%%%%%%%%%%%%%%%%%%%%%%%%%%
        &$\mathcal P_{HSS}$ &8& 0.37&&  9& 0.43&&  115&   7.10&&  $\dag$&  -\\
$5 $  &$\mathcal P_{RHSS}$&5&  0.20&&   5& 0.22&&   5&  0.19&&   9&   0.44\\
        &$\mathcal P_{REHSS}$ &5& 0.21&&     4& 0.13&&     3&  0.09&&    3&   0.08\\ \hline
%%%%%%%%%%%%%%%%%%%%%%%%%%%%%%%%%%%
      &$\mathcal P_{HSS}$ &14& 3.91&&  53&16.24&&  $\dag$& -&&  $\dag$& -\\
  $6 $&$\mathcal P_{RHSS}$&8 & 2.01&&   8& 2.05&&   8&  2.21&&   28&   8.55\\
        &$\mathcal P_{REHSS}$  &11& 3.15&&     3& 0.48&&  3&  0.43&&  3&    0.4\\ \hline
%%%%%%%%%%%%%%%%%%%%%%%%%%%%%%%%%%%%%%%
        &$\mathcal P_{HSS}$  &39& 87.20&&   $\dag$&-&&  $\dag$&-&&   $\dag$& -\\
$7 $   &$\mathcal P_{RHSS}$ &15&29.54&&  14&27.68&&  17& 35.07&&  111& 250.86\\
        &$\mathcal P_{REHSS}$  &9&16.61&&  3& 4.11&&  3&  3.81&&  3&    4.4\\ \hline
%%%%%%%%%%%%%%%%%%%%%%%%%%%%%%%%%%
        &$\mathcal P_{HSS}$  &123& 2222.41&&  -&  $\ddag$ &&- &$\ddag$&&-& $\ddag$\\
$8 $  &$\mathcal P_{RHSS}$&39& 700.77&&   31&  548.88&&   38& 690.25&&   -& $\ddag$\\
        &$\mathcal P_{REHSS}$ & 5&70.94&&   3& 35.25&&    3& 33.79&&     3&  27.73\\ \hline
\end{tabular}
\end{center}
\end{table}

In Figure 1, the eigenvalues distribution of the matrices $\mathcal{A}$ and  the preconditioned matrix $\mathcal{P}_{REHSS}^{-1} \mathcal{A}$ for the cavity problem on $32\times 32$ grid, with different  values of $\alpha$ ( $\alpha=0.1$,  $\alpha=1$, and $\alpha=10$) are displayed.  We  see that the eigenvalues of preconditioned matrices are well-clustered.

\begin{figure}[!h]
\begin{center}
\includegraphics[height=.4\paperheight]{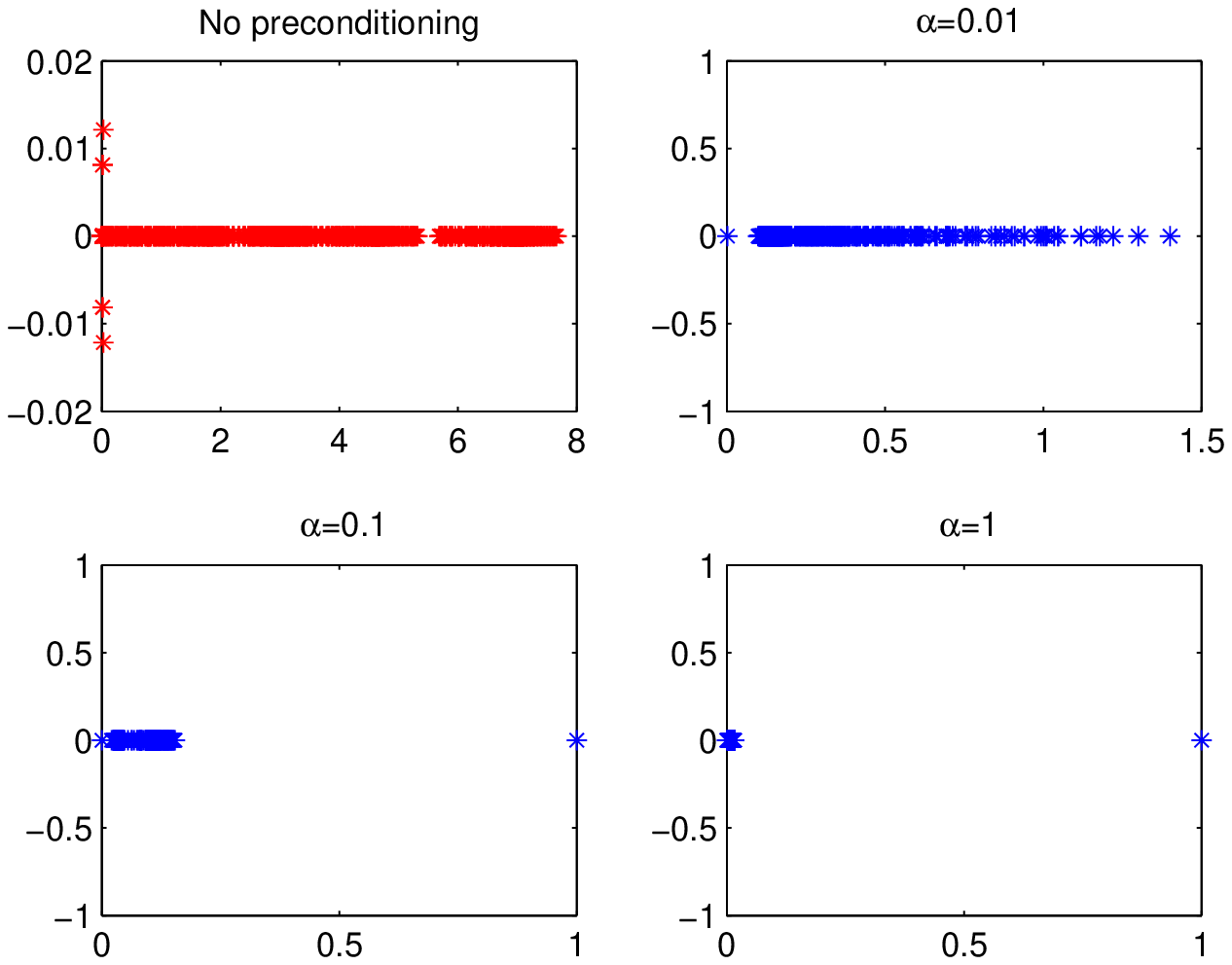}
\caption{Eigenvalues distribution of the matrices $\mathcal A$ and $\mathcal P_{REHSS}^{-1}\mathcal{A}$  for the cavity problem on $32\times 32$ grid with different  values of $\alpha$ ($\alpha=0.01,0.1,1$). }
\end{center}
\label{FigEigDis}
\end{figure}

In Figure 2,  the number of  iterations and the CPU time of  GMRES(30)  for solving the preconditioned system with the preconditioners  $\mathcal{P}_{REHSS}$, $\mathcal {P}_{HSS}$ and $\mathcal {P}_{RHSS}$ for the channel domain problem with $128\times 128$ grid for different  values of $\alpha$ are presented. As we see, for this example  the $\mathcal{P}_{REHSS}$ is superior to the preconditioners $\mathcal {P}_{HSS}$ and $\mathcal {P}_{RHSS}$, in terms of the iterations count and the CPU time.

\begin{figure}[!ht]
\begin{center}
\includegraphics[height=.35\paperheight]{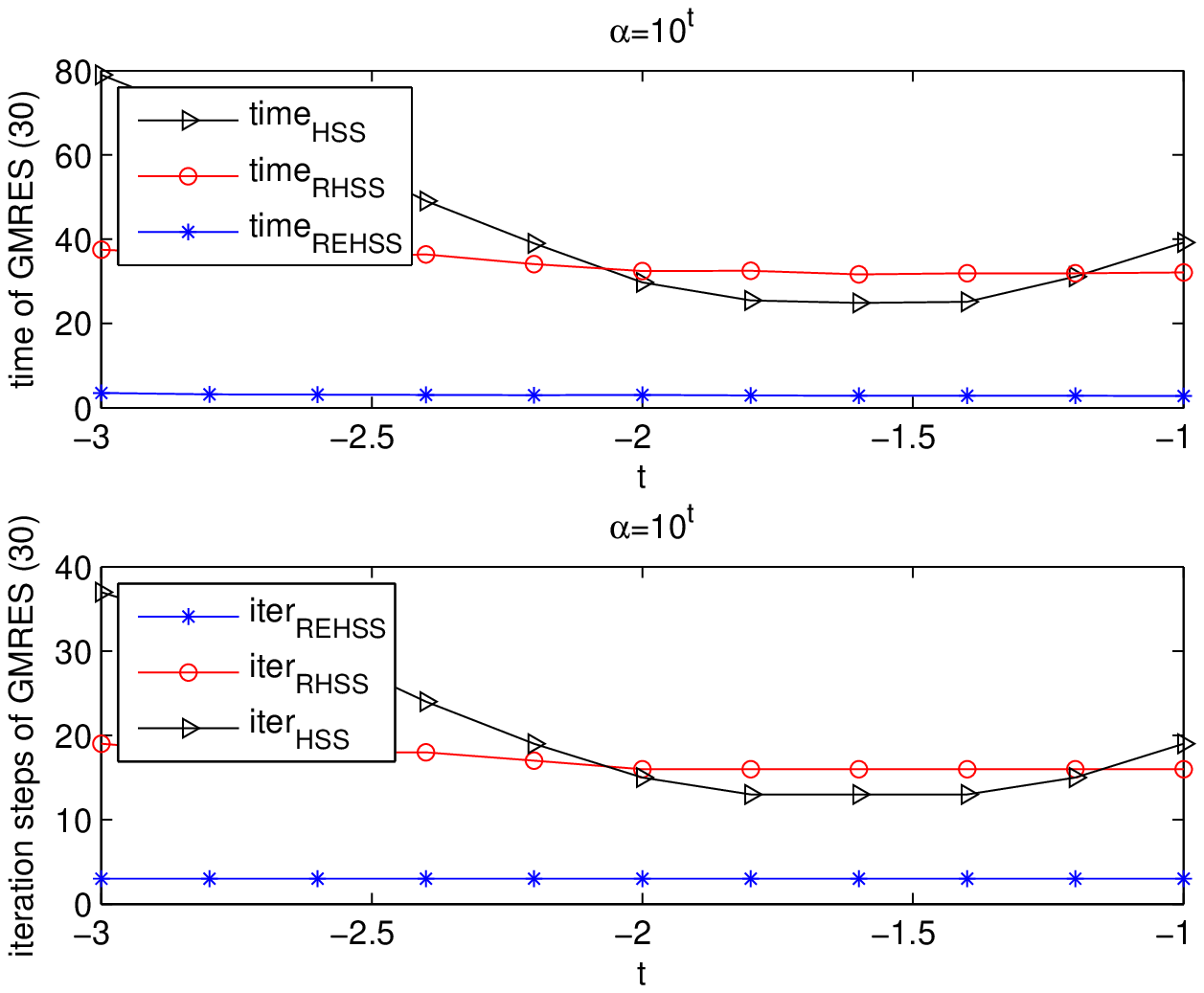}
\caption{Number of iterations and CPU time with respect to $t=\log_{10} \alpha$ for the channel domain problem with $128\times 128$ grid. }
\end{center}
%\label{FigIterCPU}
\end{figure}

\section{Concluding remarks}\label{Sec6}

We have presented a new relaxed version of the Hermitian and skew-Hermitian splitting preconditioner say REHSS for the saddle point problem (\ref{saddle}). Some properties of the preconditioner have been presented. From numerical point of  view the proposed preconditioner has been compared with two recently proposed preconditioners HSS and RHSS. Numerical results showed that the REHSS preconditioner is in general superior to the HSS and RHSS preconditioners. Moreover, the REHSS preconditioner is not very sensitive to the involving parameter.

\section*{Acknowledgements}

This work is supported by the Iran National Science Foundation (INSF) under Grant No. 93050251.
The work is also partially supported by University of Guilan.  The authors are grateful to the anonymous reviewers for their useful comments.

\end{document}